\documentclass[12pt]{amsart}

\usepackage{amsmath, amssymb, amscd, verbatim, xspace,amsthm}
\usepackage{latexsym, epsfig, color}
\usepackage[OT2,T1]{fontenc}
\DeclareSymbolFont{cyrletters}{OT2}{wncyr}{m}{n}
\DeclareMathSymbol{\Sha}{\mathalpha}{cyrletters}{"58}

\newcommand{\A}{\ensuremath{{\mathbb{A}}}}
\newcommand{\C}{\ensuremath{{\mathbb{C}}}}
\newcommand{\Z}{\ensuremath{{\mathbb{Z}}}\xspace}
\renewcommand{\P}{\ensuremath{{\mathbb{P}}}}

\newcommand{\Q}{\ensuremath{{\mathbb{Q}}}}
\newcommand{\R}{\ensuremath{{\mathbb{R}}}}
\newcommand{\F}{\ensuremath{{\mathbb{F}}}}

\newcommand{\ra}{\rightarrow}

\newcommand\Conf{\operatorname{Conf}}
\newcommand\Hom{\operatorname{Hom}}

\newcommand\Aut{\operatorname{Aut}}

\newcommand\im{\operatorname{im}}

\newcommand\Gal{\operatorname{Gal}}
\newcommand\Nm{\operatorname{Nm}}

\newcommand\Sur{\operatorname{Sur}}

\newcommand\Tr{\operatorname{Tr}}

\newcommand\tensor{\otimes}
\newcommand\isom{\simeq}
\newcommand\sub{\subset}

\newcommand\Disc{\operatorname{Disc}}

\newcommand\Spec{\operatorname{Spec}}

\newcommand\Frob{\operatorname{Frob}}

\renewcommand\O{\mathcal{O}}

\newcommand\bq{\begin{equation}}
\newcommand\eq{\end{equation}}

\newtheorem{proposition}{Proposition}[section]
\newtheorem{theorem}[proposition]{Theorem}
\newtheorem{corollary}[proposition]{Corollary}

\newtheorem{lemma}[proposition]{Lemma}

\newtheorem{conjecture}[proposition]{Conjecture}

\theoremstyle{remark}
\newtheorem{remark}[proposition]{Remark}
\usepackage[hidelinks]{hyperref}

\usepackage[all]{xy}
\usepackage{fullpage, pdfsync}

\newenvironment{definition}{\vspace{2 ex}{\noindent{\bf Definition. }}}{\vspace{2 ex}}

\newtheorem{nts}{Note to self}

\renewcommand\ss{Schur $\sigma$-}
\newcommand\pp{pro-$p$\xspace}
\newcommand\shs{Schur $\sigma$-ancestor }
\renewcommand\O{\mathcal{O}}
\newcommand{\fG}{P}

\newcommand\tH{G}
\newcommand{\un}{\operatorname{un}}
\renewcommand{\v}{\infty}
\newcommand\GG{G}
\newcommand{\CCHur}{\mathsf{CHur}}
\newcommand{\CHur}{\operatorname{CHur}}
\newcommand{\Ct}{C_2}
\newcommand{\SC}{C}

\title{Nonabelian Cohen-Lenstra Heuristics over Function Fields}

\author{Nigel Boston}
\address{Department of Mathematics\\
University of Wisconsin-Madison \\ 480 Lincoln Drive \\
Madison, WI 53706 USA}
\email{boston@math.wisc.edu}
\author{Melanie Matchett Wood}
\address{Department of Mathematics\\
University of Wisconsin-Madison \\ 480 Lincoln Drive \\
Madison, WI 53706 USA\\
and
American Institute of Mathematics\\600 East Brokaw Road\\
San Jose, CA 95112 USA}
\email{mmwood@math.wisc.edu}



\begin{document}

\begin{abstract}
Boston, Bush, and Hajir have developed heuristics, extending the Cohen-Lenstra heuristics, that 
conjecture the distribution of the Galois groups of the maximal unramified pro-$p$ extensions of imaginary quadratic number fields for $p$ an odd prime.
In this paper, we find the moments of their proposed distribution, and further prove there is a unique distribution with those moments.  Further, we show that in the function field analog, for imaginary quadratic extensions of $\F_q(t)$,
the Galois groups of the maximal unramified pro-$p$ extensions, as $q\ra \infty$, have the moments predicted by  
the Boston, Bush, and Hajir heuristics.  In fact, we determine the moments of the Galois groups of the maximal unramified pro-odd extensions of imaginary quadratic function fields, leading to a conjecture on  Galois groups of the maximal unramified pro-odd extensions of imaginary quadratic number fields.
\end{abstract}

\maketitle

\section{Introduction}
We fix an odd prime $p$ throughout the paper.  The Cohen-Lenstra
heuristics  \cite{CL84} predict the distribution of abelian $p$-groups that show up
as the $p$-primary part of the class group of an imaginary quadratic
number field as we vary the field.  In particular, there is a measure $\mu_{CL}$ on finite abelian $p$-groups, such that
$\mu_{CL}(G)>0$ for every finite abelian $p$-group $G$, that is
uniquely characterized by the fact that for any $G_1,G_2$ finite
abelian $p$-groups $ \mu_{CL}(G_1)/\mu_{CL}(G_2)={|
  \Aut(G_2)|}/{|\Aut(G_1)|}.  $
 We let $D_X$ denote the set of imaginary quadratic fields of absolute 
discriminant
less than $X$, and let $C_K$ denote the $p$-primary part of the class group of a field $K$, called the $p$-class group of $K$.
Cohen and Lenstra then conjecture the following.
\begin{conjecture}[Cohen-Lenstra, 8.1 of \cite{CL84}]\label{C:CL}
For any ``reasonable'' function $f$ on isomorphism classes of finite abelian $p$-groups, we have 
$$
\lim_{X\ra\infty} \frac{\sum_{K \in D_X} f(C_K)}{\#D_X} =\int_G f(G) d\mu_{CL}.$$

\end{conjecture}

By class field theory, the $p$-class group of a
number field $K$ is isomorphic to the Galois group $A_K$ of the
maximal abelian unramified $p$-extension of $K$.  We use this
perspective in which Cohen-Lenstra predicts the distribution of Galois
groups of such extensions to consider a generalization of the above
conjecture to nonabelian unramified extensions of imaginary quadratic
fields $K$, as follows.

Let $G_K$ be the Galois group of the maximal unramified pro-$p$ extension
of $K$, also called its \emph{$p$-class tower group}. Boston, Bush,
and Hajir \cite{BBH16} have made predictions about how often one
should expect a given group to appear as $G_K$.  Unlike $A_K$, it
turns out that $G_K$ can be infinite and this introduces new features
in the nonabelian case--for example, the measure on candidate groups
is no longer discrete.  We put a measure $\mu_{BBH}$ on the set of finitely
generated pro-$p$ groups (see Section~\ref{S:definemu} for the precise
definition), so that the conjecture of Boston, Bush,
and Hajir  is the following.

\begin{conjecture}[Boston-Bush-Hajir, cf. \cite{BBH16}]\label{C:BBH}
For any ``reasonable'' function $f$ on isomorphism classes of pro-$p$ groups, we have 
$$
\lim_{X\ra\infty} \frac{\sum_{K \in D_X} f(G_K)}{\#D_X} =\int_G f(G) d\mu_{BBH}.$$
\end{conjecture}

Of such reasonable $f$, certain are particularly interesting, and their averages $\int_G f(G) d\mu_{BBH}$ we call the \emph{moments} of the measure $\mu_{BBH}$.
 To define these $f$, first note that the $p$-class
tower group $G_K$ has a generator-inverting automorphism $\sigma$
coming from the action of $\Gal(K/\Q)$.  If $G$ and $H$ are both  profinite groups for which we have a chosen automorphism (we call both automorphisms $\sigma$), then we write $\Sur_\sigma(G,H)$ for the continuous
``$\sigma$-equivariant'' surjections from $G$ to $H$. 
The measure $\mu_{BBH}$ is supported on groups $G$ with a unique, up to conjugation, generator-inverting automorphism, which we also denote as $\sigma$. 
The average $\int_G |\Sur_\sigma(G,H)| d\mu_{BBH}$ is called the \emph{$H$-moment} of the measure $d\mu_{BBH},$ and we determine these moments. 
(See Section~\ref{S:noneq} for the simple relationship between these moments and the analog without the $\sigma$-equivariant condition.)

\begin{theorem}[Moments of $\mu_{BBH}$]\label{T:intromom}
For every finite $p$-group $H$ with a generator-inverting automorphism $\sigma$, we have
\begin{equation}\label{E:mom}
\int_G |\Sur_\sigma(G,H)| d\mu_{BBH}=1.
\end{equation}
\end{theorem}
Theorem~\ref{T:intromom} will be proven as part of 
 Theorem~\ref{T:mom} below.
Further, we show that these moments characterize the measure $d\mu_{BBH}$.

\begin{theorem}[Moments characterize $\mu_{BBH}$]\label{T:intromomdet}
If $\nu$ is a
  measure (for the $\sigma$-algebra $\Omega$ generated by groups with a fixed $p$-class $c$ quotient - these terms will be defined in \S 3) on the set of isomorphism
  classes of finitely generated pro-$p$ groups such that
$$
\int_G|\Sur_\sigma(G,H)|d\nu=1
$$
for every finite $p$-group $H$ with a generator-inverting automorphism $\sigma$, then $\nu=\mu_{BBH}.$ 
\end{theorem}

In fact, in Theorem~\ref{T:momdet} we prove a slightly stronger version of Theorem~\ref{T:intromomdet} in which we only use some of the moments.
If we take $H$ in Equation~\eqref{E:mom} to be abelian and note that under abelianization $\mu_{BBH}$ pushes forward to $\mu_{CL}$, 
 then we recover the observation of Ellenberg, Venkatesh, and Westerland \cite[Section 8.1]{EVW1} that the $A$-moments of $\mu_{CL}$ are $1$ for every abelian $p$-group $A$.  They have also shown that these $A$-moments characterize $\mu_{CL}$  \cite[Lemma 8.2]{EVW1}.  The collection of moments given by averaging $|\Sur_{\sigma}(-,H)|$ is  a fixed upper triangular transformation from the averages of $|\Hom_{\sigma}(-,H)|$.  For finite abelian groups, these latter averages are the mixed moments (of the standard invariants of the group) in the usual sense  (see \cite[Section 3.3]{Clancy15}).
 
In this paper, we prove a  theorem towards the function field analog of Conjecture~\ref{C:BBH}.
We consider the function field $\F_q(t)$, where $q$ is a prime power.
We say $K/\F_q(t)$ is \emph{imaginary quadratic} if $K$ is a degree $2$
extension of $\F_q(t)$ that is ramified at the place corresponding to
$\frac{1}{t}$, or equivalently, the smooth, projective hyperelliptic
curve corresponding to $K$ is ramified over $\infty$.  
For a quadratic
extension $K/\F_q(t)$, we let $K^{\un,\infty}$ be the maximal unramified extension of $K$ that is split completely over every place of $K$ that lies over the place $\infty$ in $\F_q(t)$, and let
$G_K^{\un,\infty}=\Gal(K^{\un,\infty}/K)$, with a generator-inverting automorphism $\sigma$ coming from the action of $\Gal(K/\F_q(t))$ (see Section~\ref{S:ff}). 
\begin{theorem}\label{T:FFqlimit}
Let $H$ be a finite odd order group with a generator-inverting
  automorphism such that the center of $H$ contains no elements fixed
  by $\sigma$ except the identity. 
  Let
  $$
 \delta_q^+ := \limsup_{m \ra\infty}  \frac{  \sum_{K\in E_m} |\Sur_\sigma(G_K^{\un,\infty},H)|}{\#E_m} 
 \textrm{  and  }   \delta_q^- := \liminf_{m \ra\infty}  \frac{  \sum_{K\in  E_m} |\Sur_\sigma(G_K^{\un,\infty},H)|}{\#E_m} ,
  $$
  where $E_m$ denotes the set of imaginary quadratic extensions $\F_q(t)$ with discriminant of norm 
  $q^{2m+2}$.
Then as $q\ra\infty$ among prime powers relatively prime to $2|H|$ and with $(q-1,|H|)=1$, we have
$$
 \delta_q^+ ,  \delta_q^- \ra 1.
$$
\end{theorem}


In light of Theorems~\ref{T:intromom} and \ref{T:intromomdet}, this is good evidence for Conjecture~\ref{C:BBH}. 
When $H$ is a $p$-group, the surjections in Theorem~\ref{T:FFqlimit} factor through the maximal pro-$p$ quotient of $G_K^{\un,\infty}$, which is analogous to the $G_K$ defined above. If we have an analogy between $\F_q(t)$ and $\Q$ for any $q$, then the $q$ limits in Theorem~\ref{T:FFqlimit} shouldn't matter, and after that limit we get agreement with the $\mu_{BBH}$ moments by Theorem~\ref{T:intromom}.  Since these moments determine a unique measure by Theorem~\ref{T:intromomdet}, that suggests Conjecture~\ref{C:BBH} for general $f$, though technically the $G_K$ do not have to be distributed according to a measure, but only a limit of measures.

Further, if we assume a vanishing conjecture on the homology of Hurwitz spaces, then under the hypotheses of Theorem~\ref{T:FFqlimit} we would in fact obtain that for $q\geq N(H)$ we have $\delta_q^+ =  \delta_q^-=1$ (see Theorem~\ref{T:withconj}).
Theorem~\ref{T:FFqlimit} suggests the following conjecture, extending Conjecture~\ref{C:BBH} from pro-$p$ groups to pro-odd groups, at least in the case of the moments.

\begin{conjecture}\label{C:prododd}
For any imaginary quadratic number field $K$, let $\mathcal{G}_K$ be the maximal pro-odd quotient of the Galois group of the maximal unramified extension
of $K$.  Then for every finite odd group $H$ with a generating-inverting automorphism
$$
\lim_{X\ra\infty} \frac{\sum_{K \in D_X} \Sur_{\sigma}(\mathcal{G}_K,H)}{\#D_X} =1.
$$
\end{conjecture}

 Bhargava \cite[Section 1.2]{BhaGeomSieve} has asked what we should expect for the average number of $H$ quotients of $G_K^{\un,\infty}$, for any $H$.
 Conjecture~\ref{C:prododd} suggests the answer for odd $H$.  (See Section~\ref{S:noneq} for the translation
 from our conjecture for $\sigma$-equivariant quotients to the consequence for more general quotients.)
 Bhargava \cite[Section 1.2]{BhaGeomSieve} has proven some intriguing moments for $H=A_3,A_4,A_5,S_3,S_4,S_5$.
 
It would be interesting to have a concrete description of an underlying measure on pro-odd groups that gives the moments on Conjecture~\ref{C:prododd}, as $\mu_{BBH}$ does in the pro-$p$ case.  However, before making a conjectural analog of Conjecture~\ref{C:BBH}, one should note it is an open question whether $\mathcal{G}_K$ is (topologically) finitely generated or not, let alone finitely presented.

In order to prove Theorem~\ref{T:FFqlimit}, in section~\ref{S:Trans}, we translate the sum of counts of surjections to a count of extensions of $\F_q(t)$ with
  certain properties.  
We then, in Section~\ref{S:EVW}, apply the recent powerful results of Ellenberg,
  Venkatesh, and Westerland \cite{EVW1,EVW2}
on homological stability of Hurwitz spaces and the components of Hurwitz spaces along with their Galois action over $\F_q$ in order to count the extensions.
A main motivation for the work of 
Ellenberg, Venkatesh, and Westerland is to prove function field analogs of Conjecture~\ref{C:CL}.  In particular, \cite[Theorem 8.8]{EVW1} gives the case of Theorem~\ref{T:FFqlimit} when $H$ is an abelian $p$-group.
The analysis of components of Hurwitz spaces in \cite{EVW2}
  gives the number of components in terms of certain group-theoretically defined
  quantities, which we compute in the cases necessary for our application.  We apply results on Hurwitz spaces from \cite{EVW1,EVW2}, the Grothendieck-Lefschetz trace formula, and our group theory computation to count $\F_q$ points of a moduli space that parametrize the relevant extensions of $\F_q(t)$.

Finally, we make some remarks on the hypotheses in Theorem~\ref{T:FFqlimit}.  The condition on the center of $H$ comes
from a technical limitation of \cite{EVW2}.
The requirement that
$(q-1,|H|)=1$ ensures that the base field does not have ``extra
roots of unity.'' The case of extra roots of unity is one in which even the Cohen-Lenstra
heuristics are expected to be wrong \cite{malle-cl} and new heuristics
have been proposed by Garton \cite{Gar15} and Adam and Malle \cite{AM15} for that case. To
the authors' knowledge, there is no work on even the Cohen-Lenstra
heuristics in the function field setting when $(q,|H|)>1$ or $2\mid q$.

\section{Background on non-abelian analogs of class groups}\label{S:ff}

Let $Q$ be a global field and $\infty$ a place of $Q$.  
In this paper, we are interested in the cases $Q=\Q$ or $\F_q(t)$ with the usual infinite place.
For a separable, quadratic
extension $K/Q$, we let $K^{\un,\infty}$ be the maximal unramified extension of $K$ that is split completely over all places of $K$ over $\infty$, and let
$G_K^{\un,\infty}=\Gal(K^{\un,\infty}/K)$.
We let $G_K$ be the maximal pro-$p$ quotient of $G_K^{\un,\infty}$.

\begin{remark}  While it looks like we've added the condition
at $\infty$ compared to the definition of $G_K$ for number fields in
the introduction, we could in fact add this condition to the
definition of $G_K$ for a quadratic number field $K$ without
effect because, for an archimedean place, unramified is the same as split
completely.
Also, if $Q=\F_q(t)$ and
$\O_K$ is the integral closure of $\F_q[t]$ in $K$, then class field
theory gives that the abelianization $(G_K^{\un,\infty})^{ab}$ is isomorphic to the
class group $Cl(\O_K)$ of ideals modulo principal ideals, so $G_K^{\un,\infty}$ is the natural function field analog of a ``non-abelian class group.''
\end{remark}

\begin{lemma}\label{L:whinertia}
If $K/Q$ is a separable, quadratic
extension, then
all inertia subgroups of $\Gal(K^{\un,\infty}/Q)$ and the decomposition group at infinity are contained in
$$\{1\} \cup\{r\in \Gal(K^{\un,\infty}/Q)\setminus G_K^{\un,\infty}\, |\, r^2=1 \}.$$
\end{lemma}
\begin{proof}
The intersection with $G_K^{\un,\infty}$ of any inertia subgroup or the decomposition group at infinity 
is trivial by the definition of $K^{\un,\infty}$, which also implies they have order at most $2$.
\end{proof}

If $Q$ is a global field and $\infty$ is a place of $Q$ such that $Q$ has no non-trivial finite extensions unramified everywhere and  split completely over $\infty$ (such as in our cases of interest $Q=\Q$ or $\F_q(t)$), we call $Q,\infty$ \emph{rational-like}.  
Then we have that 
$\{r\in \Gal(K^{\un,\infty}/Q)\setminus G_K^{\un,\infty}\, |\, r^2=1 \}$ is non-empty.
So the exact sequence
$$
1\ra G_K^{\un,\infty} \ra \Gal(K^{\un,\infty}/Q) \ra \Gal(K/Q)\ra 1,
$$
splits.
Any lift of the generator of $\Gal(K/Q)$ gives an order
$2$ automorphism of $G_K^{\un,\infty}$ by conjugation.  

\begin{proposition}\label{P:getGI}
Let $Q,\infty$ be rational-like and $K/Q$ a separable, quadratic
extension.
 The  action of an element $\tau\in \Gal(K^{\un,\infty}/Q)\setminus G_K^{\un,\infty}$ of order $2$ on $G_K^{\un,\infty}$ by conjugation inverts a set of (topological) generators of $G_K^{\un,\infty}$.
\end{proposition}

\begin{proof}
We write $\Gal(K^{\un,\infty}/Q)= G_K^{\un,\infty} \rtimes \langle \tau \rangle$. 
Let $R$ be the closed subgroup of $\Gal(K^{\un,\infty}/Q)$ generated by $\{r\in \Gal(K^{\un,\infty}/Q)\setminus G_K^{\un,\infty}\, |\, r^2=1 \}$.  From the definition, it follows that $R$ is normal.  
 So $R$ corresponds to a subfield
$M$ of $K^{\un,\infty},$ which is Galois over $Q$, and such that in $\Gal(M/Q)$ all inertia groups are trivial and the decomposition group at infinity is trivial by Lemma~\ref{L:whinertia}.  It follows that $M=Q$.  
The order $2$ elements of  $\Gal(K^{\un,\infty}/Q)\setminus G_K^{\un,\infty}$ are the $(g_i,\tau),$ for $g_i\in G_K^{\un,\infty}$ such that $g_i^\tau=g_i^{-1}.$ 
  So the words in  $\{(g_i,\tau) \, |\, g_i\in G_K^{\un,\infty}, g_i^\tau=g_i^{-1}\}$ are dense in $\Gal(K^{\un,\infty}/Q).$ An element of $G_K^{\un,\infty}$ equivalent to one of these words is a word in the symbols
$\{ g_i\in G_K^{\un,\infty}\, |\, g_i^\tau=g_i^{-1}\}$, and such elements are a dense subgroup of $G_K^{\un,\infty}$.  Thus  the set $\{ g_i\in G_K^{\un,\infty}\, |\, g_i^\tau=g_i^{-1}\}$ topologically generates $G_K^{\un,\infty}$.
\end{proof}

In light of Proposition~\ref{P:getGI}, we pick a lift  $\tau$ of the generator of $\Gal(K/Q)$
to $\Gal(K^{\un,\infty}/Q)$ and let conjugation by $\tau$ be our chosen 
generator-inverting automorphism $\sigma$ of $G_K^{\un,\infty}$.  Further, the Schur-Zassenhaus theorem (Prop.2.3.3 of \cite{Wil98}) guarantees that all the lifts of the generator of $\Gal(K/Q)$ to the pro-$p$ quotient $G_K$ of $\Gal(K^{\un,\infty}/Q)$ (or the pro-odd quotient) are conjugate.
Thus for an odd finite group $H$ with automorphism $\sigma$, we then have that $|\Sur_{\sigma}(G_K^{\un,\infty},H)|$ does not depend on the choice of $\tau$.

\section{Boston-Bush-Hajir heuristics: background and notation}\label{S:definemu}
Koch and Venkov \cite{KV75} have shown that for an imaginary quadratic extension $K/\Q$, the group $G_K$ satisfies certain properties we will now outline.  For a pro-$p$ group $G$, let $d(G):=\dim_{\Z/p\Z}H^1(G,\Z/p\Z)$ and 
$r(G):=\dim_{\Z/p\Z}H^2(G,\Z/p\Z).$   These are, respectively, the generator rank and the relation rank of $G$ as a pro-$p$ group.  
For a pro-finite group $G$, we define a \emph{GI-automorphism} of $G$ to be a  $\sigma\in\Aut(G)$ such that 
$\sigma$ acts as inversion on a set of (topological) generators.  For a pro-$p$ group, this is equivalent to requiring that 
$\sigma^2=1$, which $\sigma$ are called \emph{involutions}, and
$\sigma$ acts as inversion on the abelianization of $G$ \cite{Bos91}. 

\begin{definition}
 A \emph{Schur-$\sigma$} group is a finitely generated pro-$p$ group $G$ with finite abelianization such that
\begin{enumerate}
 \item $d(G)=r(G)$ (then called just the \emph{rank} of $G$)
\item $G$ admits a GI-automorphism.
\end{enumerate}
\end{definition}
Koch and Venkov \cite{KV75} have shown that for an imaginary quadratic extension $K/\Q$, the group $G_K$ is a 
Schur-$\sigma$ group. The groups $G_K$ we are considering in the function field case are also Schur-$\sigma$ groups when $p\nmid q-1$. This follows by class field theory, Proposition~\ref{P:getGI} above, and the upper bound on $r(G_K)-d(G_K)$, namely $0$, due to
Shafarevich, given as Theorem 2.2 in \cite{HM01}. Note that $r(G_K)-d(G_K) \geq 0$ since $G_K^{ab}$ is finite and so the upper bound of $0$ yields $r(G_K)-d(G_K)=0$.

We will put a measure on the set of isomorphism classes of \ss groups
in order to state the Boston-Bush-Hajir heuristics.  For this, we
first need to define a $\sigma$-algebra (in the sense of measure
theory--not our automorphism $\sigma$) on this set.  Since
many infinite \ss groups are expected to occur as $G_K$ with density
$0$, it makes sense to focus on certain finite quotients of these
groups.

Any pro-$p$ group $G$ has a \emph{lower $p$-central series} defined as
$P_0(G):=G$ and for $n\geq 0$, we let $P_{n+1}(G)$ be the closed
subgroup generated by $[G,P_n(G)]$ and $P_n(G)^p$.  The groups
$P_0(G)\geq P_1(G) \geq P_2(G)\geq \cdots$ form a descending chain of
characteristic subgroups of $G$ called the lower $p$-central series.
The \emph{$p$-class} of a finite $p$-group $G$ is the smallest $c\geq
0$ for which $P_c(G)=\{1\}$.  Note that for a finitely generated pro-$p$ group
$G$, the successive quotients $P_n(G)/P_{n+1}(G)$ are finite abelian
groups of exponent $p$, and so in particular, if $P_c(G)=\{1\}$, then
$G$ must be finite. The lower $p$-central series and $p$-class can be
thought of as analogous to the lower central series and nilpotency
class, respectively.  Note that $P_1(G)$ is also the Frattini subgroup
$\Phi(G)$.

For a \pp group $G$, we define $Q_c(G):=G/P_c(G)$, the maximal
quotient of $G$ with $p$-class at most $c$.  So $Q_c(G_K)$ is the
Galois group of the maximal unramified $p$-extension of $K$ among
extensions of Galois group with $p$-class at most $c$.  Note that since
a \ss group $G$ (such as $G_K$) is finitely generated, we have that
$Q_c(G)$ is finite.  It may be that $Q_c(G)$ has $p$-class strictly
less than $c$: certainly when $G$ itself has $p$-class strictly less
than $c$, this happens, but in fact since the subquotients of the
lower $p$-central series for $G$ and for $Q_c(G)$ are the same up to
index $c$, this is the only way it can happen.

Let $\Omega$ be the $\sigma$-algebra on the set of isomorphism classes
of \ss groups generated by the sets
\begin{equation}\label{E:defineSigma}
\{G | Q_c(G)\isom \fG \}
\end{equation}
for each finite $p$-group $\fG$ and fixed $c$.  For example, we can fix a \ss group $G_0$ and
take the intersection over all $c$ of 
$\{G | Q_c(G)\isom Q_c(G_0)\}$
to see that $\Omega$ contains the singleton set containing the class of $G_0$.

We will next define a measure on the set of isomorphism classes of \ss
groups for a $\sigma$-algebra containing $\Omega$.  Any
\ss group of rank $g$ can be presented as a quotient of the free \pp
group $F_g$ on $g$ generators $x_1,\dots,x_g$ (with
GI-automorphism $\sigma(x_i)=x_i^{-1}$) by $g$
relations chosen from $X=\{s\in \Phi(F_g) | \sigma(s)=s^{-1}\}$.  Since $X$ is a closed
subset of the profinite group $F_g$, we have a natural profinite probability measure
$\mu$ on $X$ from the limit of the uniform measures on finite quotients of $F_g$, on the
$\sigma$-algebra generated by fibers of these quotients.

The Boston-Bush-Hajir probability measure $\mu_{BBH}$ will be given by
randomly selecting such relations.  However, this only gives a measure
for a fixed rank $g$ of \ss groups.  Since, however, the rank of a \ss
group is the rank of its abelianization (in fact, of the quotient of
the abelianization $G/\Phi(G)$, by the Burnside basis theorem), we can use
the Cohen-Lenstra heuristics to predict how often each rank $g$
occurs.  Let
$$
\mu_{CL}(g):=\sum_{G \textrm{ fin. ab., rk $g$ $p$-gp}} 
\mu_{CL}(G)= p^{-g^2}\prod_{k=1}^{g}(1-p^{-k})^{-2}\prod_{i=1}^{\infty}(1-p^{-i}).
$$

The above formula is from Theorem 6.3 of \cite{CL84}. Let $A$ be a set of isomorphism classes of rank $g$ \ss groups.  Then we define
$$
\mu_{BBH}(A):=\mu_{CL}(g)\mu(\{(r_1,\dots,r_g)\in X^g \,|\, F_g/\langle\langle r_1,\dots,r_g\rangle\rangle \in A \})
,
$$
whenever $\{(r_1,\dots,r_g)\in X^g \,|\, F_g/\langle \langle
r_1,\dots,r_g\rangle\rangle \in A \}$ is measurable, where the double
angle brackets denote the closed normal subgroup generated by the
elements.  We can think of this measure as generating a random group
by picking a rank $g$ according to the Cohen-Lenstra measure and then
independently creating a random \ss group of rank $g$ by taking the
quotient of the free \pp group $F_g$ on $g$ generators by $g$ randomly
chosen relations in $X$.  Note that this process does not necessarily
produce a \ss group, as there may be redundancy among the relations
and so the resulting group may not have relation rank $g$.  However,
such redundancy happens with probability $0$ (the abelianization would
be infinite, and as noted by Friedman-Washington
\cite{FW89}, this occurs with zero probability under $\mu_{CL}$, which is
induced on abelianizations from $\mu_{BBH}$ \cite[Theorem 2.20]{BBH16}).

Let $X_c=\{s\in \Phi(Q_c(F_g)) | \sigma(s)=s^{-1}\}$.  Note that $X_c$ is a finite set and has a uniform
discrete probability measure $\mu_c$ that pulls back to $\mu$ on $X$.  
If $\fG$ is a fixed finite $p$-group with $d(\fG)=g$, we define $\mu_{BBH,c}(\fG):=
\mu_{BBH}(\{G\,|\,Q_c(G)\isom \fG\})$, and then
$$
\mu_{BBH,c}(\fG)=\mu_{CL}(g)
\mu_c(\{(r_1,\dots,r_g)\in X_c^g \,|\, Q_c(F_g)/\langle \langle r_1,\dots,r_g\rangle \rangle\isom \fG\}).
$$
In particular
$\{G | Q_c(G)\isom \fG \}$ is measurable for $\mu_{BBH}$.

If $\fG\isom Q_c(G)$ for some \ss group $G$, we call $\fG$ a \emph{\shs
  group}. Note that a \shs
group is necessarily a finite $p$-group with a GI-automorphism (though
these conditions are not sufficient).  The \shs groups are exactly
those presented as $Q_c(F_g)/\langle \langle
r_1,\dots,r_g\rangle\rangle$ for some $r_1,\dots,r_g\in X_c$.  This is
because one can choose an irredundant lift of the relations from $X_c$
to $X$ to give a \ss group \cite{BBH16}. In particular, for any \shs group $G$ of
$p$-class $c$, we have that $\mu_{BBH,c}(G)>0.$

\subsection{Choice of GI-automorphisms}\label{S:chooseinv}
It might seem strange at first that we do not include the choice of
GI-automorphism with our data of a \ss group or \shs group.  However,
we have the following proposition.

\begin{proposition}[Section 1.3 of \cite{Hal34}]\label{P:ric}
  Any $2$ GI-automorphisms of a finitely generated pro-$p$ group $G$
  are conjugate in $\Aut(G).$
\end{proposition}

If $G$ and $H$ are finitely generated pro-$p$ groups, we define
$\Sur_{\sigma}(G,H)$ to be the continuous surjections from $G$ to $H$
that take some particular choice of GI-automorphism for $G$ to some
particular choice of GI-automorphism for $H$.  We define
$\Aut_\sigma(G)$ similarly.  These definitions of course depend on the
particular choice of GI-automorphisms, but in this paper we will be
concerned mostly with the size of these sets, and by
Proposition~\ref{P:ric} their sizes do not depend on these choices.

\subsection{Choice of generators}
The description of $\mu_{BBH}$ above actually gives a finer measure on
the set of isomorphism classes of \ss groups \emph{with} a choice of
GI-automorphism and minimal generating set inverted by that
automorphism.  We will later take advantage of this generating set,
though for simplicity we do not introduce notation for this finer
measure.

\section{Boston-Bush-Hajir moments}\label{S:moments}

We now determine the
moments of the measure $\mu_{BBH}$ as stated in Theorem~\ref{T:intromom}.
\begin{theorem}[Moments of $\mu_{BBH}$]\label{T:mom}
  Let $H$ be a finite $p$-group of $p$-class $c$ with a GI-automorphism $\sigma$. Then
$$ \int_G |\Sur_{\sigma}(G,H)| d\mu_{BBH}=\sum_{G \textrm{ \shs of $p$-class $c$}} \mu_{BBH,c}(G)| \Sur_\sigma(G,H)|=1 .$$
\end{theorem}

Note the hypothesis that $\sigma$ is $GI$ on $H$ does not place any real restriction, because if we have a surjection $G\ra H$ that takes a GI-automorphism $\sigma_G$ on $G$ to any automorphism $\sigma_H$ on $H$, then $\sigma_H$ must also be GI.

Let $H$ be a finite $p$ group with an order $2$ automorphism
$\sigma$. We write $Z(H)=\{g\in H| \sigma(g)=g\}$ and $Y(H)=\{g\in H|
\sigma(g)=g^{-1}\}$.  This notation implicitly depends on $\sigma$.
We now prove several lemmas that will be used in the proof of
Theorem~\ref{T:mom}.
\begin{lemma}
  Let $G$ be a finite $p$-group with an order $2$ automorphism
  $\sigma$. Then $|G|=|Y(G)||Z(G)|$.
\end{lemma}
\begin{proof}
This is Theorem 3.5 (p.180) of Chapter 5 of \cite{Gor07}.
\end{proof}

\begin{lemma}
  Let $G$ and $H$ be finite $p$-groups, each with an order $2$ automorphism
  $\sigma$, and let $\phi:G\ra H$ be a $\sigma$-equivariant surjection.
  Then $\phi : Z(G)\ra Z(H)$ is a surjection.
\end{lemma}

\begin{proof}
Associated to the exact sequence $1 \ra \ker(\phi) \ra G \ra H \ra 1$ is the exact sequence
$$ \dots \ra H^0(\langle\sigma\rangle,G) \ra H^0(\langle\sigma\rangle,H) \ra H^1(\langle\sigma\rangle,\ker(\phi)) \ra \dots $$  
The first and second terms are $Z(G)$ and $Z(H)$ respectively. The last term is $H^1(\langle\sigma\rangle,\ker(\phi))$, which vanishes by Schur-Zassenhaus since $p$ is odd.

\end{proof}

\begin{lemma}\label{L:kersize}
  Let $G$ and $H$ be finite $p$-groups, each with an order $2$ automorphism
  $\sigma$, and let $\phi:G\ra H$ be a $\sigma$-equivariant surjection
  with kernel $K$.  Then $Z(K) = K\cap Z(G)$ and $Y(K)=K \cap Y(G)$, and
  $|Y(K)|=|Y(G)|/|Y(H)|$.
\end{lemma}
\begin{proof}
 The first two claims are clear.  Using the above two lemmas, we then observe: 
$$
|Y(K)|=\frac{|K|}{|Z(K)|}=\frac{|G|/|H|}{|Z(G)|/|Z(H)|}=\frac{|Y(G)|}{|Y(H)|}
$$

which proves the final claim.
\end{proof}

\begin{lemma}
  Let $H$ be a finite $p$-group with GI-automorphism $\sigma$.  Then
  the elements of $Y(H)$ are equidistributed in $H/\Phi(H)$.  That is,
  any two cosets in $H$ of $\Phi(H)$, when intersected with $Y(H)$ have
  the same number of elements.
\end{lemma}
\begin{proof}
  We consider the maps of sets $f:H \ra Y(H)$ given by
  $f(g)=g^{-1}\sigma(g)$ and $\pi: Y(H) \ra H/\Phi(H)$ the composition
  of the inclusion and quotient maps $Y(H) \ra H \ra H/\Phi(H)$.  

  Then the composition $\pi f : H\ra H/\Phi(H)$ sends $g\mapsto
  g^{-2}$ since $\sigma$ acts by inversion on $H/\Phi(H)$.  This is a
  homomorphism since $H/\Phi(H)$ is abelian, and a surjection since
  $H/\Phi(H)$ has odd order.  Thus the fibers of $\pi f$ are of equal
  size.  Further, the fibers of $f$ are cosets of $Z(H)$ and thus are
  also of equal size.  Also, since for any $g \in H$, $g\Phi(H) \cap Y(H)
  = \pi^{-1}(g)$, it suffices to show the fibres of $\pi$ have equal
  sizes, which now follows.
\end{proof}

\begin{lemma}\label{L:countsur}
 Let $H$ be a finite $p$-group of generator rank $r$ with a GI-automorphism $\sigma$.  Then
$$
|\Sur_\sigma(F_d,H)|=\frac{|Y(H)|^d(p^d-p^{r-1})\cdots(p^d-1)}{p^{dr}}
$$
\end{lemma}
\begin{proof}
  A homomorphism $F_d\ra H$ is $\sigma$-equivariant if and only if it sends each of the
  $d$ generators of $F_d$ to an element of $Y(H)$, and so there are
  $|Y(H)|^d$ such maps.  By the Burnside basis theorem, such a
  homomorphism is surjective if and only if its composition with the quotient map is surjective to
  $H/\Phi(H)$.  Since the elements of $Y(H)$ are equidistributed in
  $H/\Phi(H)$, the proportion of $\sigma$-equivariant homomorphisms
  $F_d\ra H$ that are surjective is the same as the proportion of
  $d$-tuples from $H/\Phi(H)\isom (\Z/p\Z)^r$ that span this
  $\Z/p\Z$-vector space, which is easily computed to be
  $(p^d-p^{r-1})\cdots(p^d-1)/p^{dr}.$
\end{proof}

\begin{proof}[Proof of Theorem~\ref{T:mom}]
  Since a surjection from $G$ to $H$ factors through $Q_c(G)$, we see
  that $f(G) = |\Sur_{\sigma}(G,H)|$ is in fact a measurable function
  and that the first equality is by definition of the two measures.

  Let $H$ have generator rank $r$.  The random group $G$ is
  constructed first by picking a random generator rank $d$ for $G$
  according to the Cohen-Lenstra measure, and then taking a random
  quotient of $F_d$.  Certainly, any surjection $G\ra H$ lifts
  uniquely to a surjection $F_d \ra H$.  From Lemma~\ref{L:countsur}
  we see there are $|Y(H)|^d(p^d-p^{r-1})\cdots(p^d-1)/p^{dr}$
  $\sigma$-equivariant surjections $F_d \ra H$. A surjection $\phi
  :F_d \ra H$ factors through $G$ if and only if the $d$ random
  relations in $Y(\Phi(F_d))$ that present $G$ are in $\ker(\phi)$,
  the probability of which we now compute.  Since $H$ is $p$-class
  $c$, we may equivalently take the random relations in
  $Y(\Phi(F_d)/P_c(F_d)).$

  Let $F:=F_d/P_c(F_d)$.  The probability that a random relation in
  $X_c=Y(\Phi(F))$ is in $\ker(\phi)$ is $|\ker(\phi)\cap
  Y(\Phi(F))|/|Y(\Phi(F))|.$ Applying Lemma~\ref{L:kersize} to the
  surjection $\phi: \Phi(F) \ra \Phi(G)$, we see that $|\ker(\phi)\cap
  Y(\Phi(F))|/|Y(\Phi(F))|=|Y(\Phi(G))|^{-1}.$ Also, applying
  Lemma~\ref{L:kersize} to the quotient $G \ra G/\Phi(G)$, we have
  that $|Y(\Phi(G))|=|Y(G)|/p^r$, since $\sigma$ acts on all of
  $G/\Phi(G)$ by inversion.  Thus, the probability that $d$ random
  relations are in $\ker(\phi)$, and so the map $\phi$ factors through the
  random $G$, is $p^{dr}/|Y(H)|^d.$ 

  Multiplying by the number of $\sigma$-equivariant surjections $F_d
  \ra H$, we find that among generator rank $d$ groups $G$, the
  expected number of $\sigma$-equivariant surjections to $H$ is
  $(p^d-p^{r-1})\cdots(p^d-1)$, which is the number of surjections
  from a rank $d$ abelian $p$-group to $(\Z/p\Z)^r$.  Thus the
  expected number of $\sigma$-equivariant surjections is $\sum_{d\geq
    0} \mu_{CL}(d) (p^d-p^{r-1})\cdots(p^d-1)=\sum_A \mu_{CL}(A)
  |\Sur(A,(\Z/p\Z)^r)|=1$, by the moments formula for the
  Cohen-Lenstra measure.
\end{proof}

In fact, we will see in Theorem~\ref{T:momdet} that the moments
where $H$ is a \shs group characterize $\mu_{BBH}$ as a measure on
$\Omega$.  
At each $p$-class, showing the moments characterize the measure amounts to inverting an 
infinite-dimensional matrix.  Our method to invert this matrix can be seen as a generalization of the method of \cite[Lemma 8.2]{EVW1},
which proves that the moments characterize the Cohen-Lenstra measure on finite abelian $p$-groups.  
First we need an
infinite-dimensional linear algebra lemma, since our infinite matrices are not quite as simple as those in \cite[Lemma 8.2]{EVW1}. 

\begin{lemma}\label{L:la}
Let $a_{i,j}$ be non-negative real numbers indexed by pairs of natural numbers $i,j$, such that for all
$i$ we have $a_{i,i}=1$, and also $\sup_i \sum_{j} a_{ij}<2$.  Let $x_j,y_j$ be non-negative reals
indexed by natural numbers $j$.  If for all $i,$
$$
\sum_j a_{i,j} x_j=\sum_j a_{i,j}y_j =1,
$$ 
then $x_j=y_j$ for all $j$.
\end{lemma}

\begin{proof}
Note that  $x_i = a_{ii} x_i \leq \sum_j a_{i,j} x_j \leq  1$.  Similarly $0 \leq y_i \leq 1$.  Let $d_i = x_i - y_i$.
Let $a = \sup_i \sum_j a_{ij} < 2$.  Let $s =\sup_i |d_i|$, so $0 \leq s \leq 1$.
For each $i$, we have $\sum a_{ij} d_j = 0$, so $d_i = - \sum_{j\ne i} a_{ij} d_j$.
So, $|d_i| \leq  \sum_{j\ne i} a_{ij} |d_j|$. Taking the supremum over $i$ yields $s \leq  (a-1) s$.
Since $a-1 < 1$, so $s=0$.  Thus $x_i = y_i$ for all $i$.

\end{proof}

Next, we will prove a formula for $\mu_{BBH}(\{ G \ |\ Q_c(G)\isom \fG\})$ for a given \shs group $\fG$.
The formula combines Theorems 2.25 and 2.29 of \cite{BBH16}, which are subject to a further conjecture called KIP, but we prove below that the combined formula is not conjectural.  For the formula, we will need one further invariant of $p$-groups.  For a finite $p$-group $G$ of $p$-class $c$ presented as $F/R$, where $F$ is a free group of $d(G)$ generators, then $h(G)$ is defined to be the dimension of the quotient of $R$ by the topological closure of the subgroup $R^p[F,R]P_c(F)$ (by \cite{OBrien} and \cite[Remark 2.4]{BBH16} the quantity does not depend on the choice of presentation).

Alternatively, the $p$-groups of $p$-class $\le c$ form a variety of groups whose free objects are precisely the groups $Q_c(F_d)$. For a group $G$ in this variety, we can let $h_c(G)$ be the number of relators required to present $G$ in this variety. 
If $G$ is $p$-class $c$, then $h_c(G)=h(G)$ and if $G$ is $p$-class smaller than $c$, then $h_c(G) = r(G)$.

\begin{lemma}\label{L:BBHformula}
Fix a $c$.  Let $g=d(G)$ and $h=h_c(G)$. We have
$$
\frac{\mu_{BBH}(\{ G \ |\ Q_c(G)\isom \fG\})}{\mu_{CL}(g)}=
\frac{p^{g^2}}{|\Aut_\sigma(G)|}\prod_{k=1}^g (1-p^{-k})  \prod_{k=1+g-h}^g (1 - p^{-k})
$$


\end{lemma}
\begin{proof}


Let $F_c=Q_c(F_g)$.
We need to compute the sum of the probabilities that a given $g$-tuple of relations $v \in X_c^g$ generates $\overline{R}$ as a normal subgroup of $F_c$, where $\overline{R}$ runs over all normal subgroups of $F_c$ with quotient $G$. The key thing to note here is that since each element of $X_c$ is inverted by $\sigma$, any subgroup
generated by elements of $X_c$ is $\sigma$-invariant, as is the normal closure of such a subgroup. 
Thus if $\overline{R}$ is a normal subgroup of $F_c$ that is not $\sigma$-invariant, then the probability that is  generated as a normal subgroup by relations from $X_c$ is $0$.
In \cite{BBH16}, the conjectural property KIP (kernel invariance property) was assumed to ensure that every normal subgroup with quotient $G$ is $\sigma$-invariant.  We do not assume this, since by the above remark we can restrict our attention to the set of $\sigma$-invariant normal subgroups with quotient $G$.

The number of $\sigma$-invariant normal subgroups of $F_c$ with quotient $G$ is $|\Sur_\sigma(F_c,G)|/|\Aut_\sigma(G)|$, by counting the quotient maps and dividing by how often maps give isomorphic quotients.
(There are similarly $|\Sur(F_c,G)|/|\Aut(G)|$ normal subgroups with quotient $G$, but if there are any that are not $\sigma$-invariant we have already seen they have $0$ probability of being generated by our relations in $X_g$.) The probability that a $g$-tuple of relations $v \in X_c^g$ generates a $\sigma$-invariant $\overline{R}$ as a normal subgroup can be computed by the earlier methods of \cite{BBH16}. We give a slightly alternative treatment here.

First note that by Lemma 4.6, $|\Sur_\sigma(F_c,G)| = |Y(G)|^g \prod_{k=1}^g (1-p^{-k})$, since every such surjection from the free pro-$p$ group $F_g$ on $g$ generators factors through $F_c$. As for the probability that $v \in X_c^g$ normally generates $\overline{R}$, this happens if and only if its image generates the  $\F_p$-vector space $V = \overline{R}/\overline{R}^\ast$, where $R$ is the preimage of $\overline{R}$ in $F_g$, $R^\ast$ is the topological closure of $R^p[F_g,R]$, and $\overline{R}^\ast = P_c(F_g)R^\ast/P_c(F_g)$ \cite[Proposition 2.8]{gruenberg1976relation}.
When $G$ is $p$-class $c$, the dimension of $V$ is $h$ (by definition of $h$).  
When $G$ is $p$-class $<c$, we have $P_{c-1}(F_g)\sub R$ and so $P_c(F_g)$ is a subgroup of $R^*$.  Then $V=R/R^*$, which has dimension $r(G)$.  Let $s=\dim V$, which we have just determined in each case.
The number of $g$-tuples generating $V$ is $\prod_{k=1}^s (p^g - p^{s-k})$ and so we just need the size of the intersection of $X_c$ with a fiber of the quotient map $r: \overline{R} \rightarrow V$. 

We claim each of these has $|\overline{R}^\ast|/|Z(\overline{R})|$ elements. This follows by considering the map $f$ of Lemma 4.5, defined by $f(g)=g^{-1}\sigma(g)$. Since $V$ is abelian, $f \circ r = -2r$, whose fibers have the same size as those of $r$, namely $|\overline{R}^\ast|$, since $p$ is odd. On the other hand, $f \circ r = r \circ f$, the size of the fibers of which are the size of those of $r$ times those of $f$. This latter term is $|Z(\overline{R})|$ by Lemma 4.2. Putting these facts together establishes the claim.

To recap, the desired measure is the sum over $|Y(G)|^g \prod_{k=1}^g (1-p^{-k})/|\Aut_\sigma(G)|$ terms of the number of $v$ in $X_c^g$ normally generating each $\overline{R}$, which we just found to be $\prod_{k=1}^s (p^g - p^{s-k})( |\overline{R}^\ast|/|Z(\overline{R})|)^g$, divided by the total number of $v$, namely $|X_c|^g$. In other words,
$$ \prod_{k=1}^s (p^g - p^{s-k})\prod_{k=1}^g (1-p^{-k})  
\frac{
(|\overline{R}^\ast|/|Z(\overline{R})|)^g|Y(G)|^g}{|\Aut_\sigma(G)||X_c|^g} $$

It remains to show that $|Y(G)||\overline{R}^\ast|/(|Z(\overline{R})||X_c|) = p^{g-s}$. This follows from Lemma 4.4, which says that $|Y(F_c)| = |Y(G)||Y(\overline{R})|$ and $|Y(F_c)| = |Y(\Phi(F_c)) ||Y(F_c/\Phi(F_c))| = |X_c|p^g$. Thus, $|X_c| =  |Y(G)||Y(\overline{R})|p^{-g}$. Combining this with $|Y(\overline{R})||Z(\overline{R})|=|\overline{R}|$ (Lemma 4.2) and $|\overline{R}|/|\overline{R}^\ast| = p^s$ gives the result.
\end{proof}

\begin{theorem}[Moments characterize $\mu_{BBH}$]\label{T:momdet}
Let $\nu$ be a measure on $\Omega$ such that for every
\shs group $H$,
$$ \int_G |\Sur_{\sigma}(G,H)| d\nu=1 .$$
Then $\nu=\mu_{BBH}$.
\end{theorem}

Note that \shs groups are a proper subset of finite
$p$-groups with GI-automorphisms, so this theorem does not require all of the moments determined in Theorem~\ref{T:mom}.  

\begin{proof}
  By Carath\'eodory's theorem, a measure $\nu$ on $\Omega$ is determined by the measures $\nu(\{ G
  \ |\ Q_c(G)\isom S\})$ for all \shs groups $S$. 
If $G$ is a Schur  $\sigma$-group, then $Q_c(G)$ is either a \shs group of $p$-class $c$ or a 
  Schur  $\sigma$-group of $p$-class $<c$.  (This is because if $Q_c(G)$ is $p$-class $<c$ then $Q_c(G)=G$.)
  Let $\mathcal{S}$ be the set of isomorphism classes of groups that are either a \shs group of $p$-class $c$ or a 
  Schur  $\sigma$-group of $p$-class $<c$. 
  
For $H$  a \shs group of $p$-class $c$,  
    we have that
$$
\sum_{S \in \mathcal{S}} \nu(\{ G \ |\ Q_c(G)\isom S\}) |\Sur_\sigma(S,H)|=1
$$
and 
$$
\sum_{S \in \mathcal{S}} \mu_{BBH}(\{ G \ |\ Q_c(G)\isom S\}) | \Sur_\sigma(S,H)|=1.
$$We can index $\mathcal{S}$ by natural numbers $S_1,S_2,\dots.$ We then
apply Lemma~\ref{L:la} with
$$
a_{i,j}=\frac{|\Sur_\sigma(S_j,S_i)|}{|\Aut_\sigma (S_j)|}
$$
and $x_j=\nu(\{ G \ |\ Q_c(G)\isom S_j\})|\Aut_\sigma (S_j)|$ and
$y_j=\mu_{BBH}(\{ G \ |\ Q_c(G)\isom S_j\})|\Aut_\sigma (S_j)|$, which
will prove the proposition.  We must verify that $\sum_{j} a_{i,j}<2$.
 
 Using the explicit formulae for $\mu_{CL}(d)$ (from \cite{CL84}) and for $\mu_{BBH}$ (from Lemma~\ref{L:BBHformula}),
 we have that

 \begin{align*}
&\mu_{BBH}(\{ G \ |\ Q_c(G)\isom S_j\})= \frac{\mu_{CL}(d(S_j))p^{d(S_j)^2}
}{|\Aut_\sigma(S_j)|} \prod_{k=1}^{d(S_j)}(1-p^{-k})
\prod_{k=1+d(S_j)-h_c(S_j) }^{d(S_j)}(1-p^{-k})\\
&= \frac{\prod_{k\geq 1}(1-p^{-k})\prod_{k=1}^{d(S_j)}(1-p^{-k})^{-2}}{|\Aut_\sigma(S_j)|} \prod_{k=1}^{d(S_j)}(1-p^{-k})
\prod_{k=1+d(S_j)-h_c(S_j) }^{d(S_j)}(1-p^{-k})\\
&= \frac{1}{|\Aut_\sigma(S_j)|} \prod_{k\geq 1}(1-p^{-k})\prod_{k=1}^{d(S_j)}(1-p^{-k})^{-1}
\prod_{k=1+d(S_j)-h_c(S_j) }^{d(S_j)}(1-p^{-k}).
\end{align*}
When $S_j$ is $p$-class $c$, we have that $h_c(S_j)=h(S_j)$, and since
$S_j$ is a \shs, it is $Q_c(G)$ for some Schur $\sigma$-group $G$.  Since $r(G)=d(G)=d(S_j)$, and  
$r(G)\geq h(S_j)$ \cite[Proposition 2]{BostonNover}, we have $d(S_j)\geq h_c(S_j).$
When $S_j$ is a Schur $\sigma$-group, we have that $h_c(S_j)=r(S_j)=d(S_j)$.  In either case, we conclude that
\begin{align*}
\mu_{BBH}(\{ G \ |\ Q_c(G)\isom S_j\})\geq \frac{1}{|\Aut_\sigma(S_j)|} \prod_{k\geq 1}(1-p^{-k}).
\end{align*} 
 For all $p\geq 3,$ we have that $\prod_{k\geq 1}(1-p^{-k})>.53$ and so
$$
\frac{1}{|\Aut_\sigma(S_j)|}<1.9 \mu_{BBH}(\{ G \ |\ Q_c(G)\isom S_j\}) .
$$ 
 Thus,
\begin{align*}
\sup_i \sum_{j} a_{i,j} &= \sup_i \sum_{j}   \frac{|\Sur_\sigma(S_j,S_i)|}{|\Aut_\sigma (S_j)|}\\
&\leq 1.9 \sup_i  \sum_{j}  \mu_{BBH}(\{ G \ |\ Q_c(G)\isom S_j\})|\Sur_\sigma(S_j,S_i)|\leq 1.9.
\end{align*}
\end{proof}


\section{Moments as an extension counting problem}\label{S:Trans}

Let $Q$ be a global field with a choice of place ${\v}$.  (We are mainly interested in $Q=\Q$ or $\F_q(t)$ with the usual infinite place.)
We fix a separable closure $\bar{Q}_{\v}$ of the completion $Q_{\v}$.  Then, inside $\bar{Q}_{\v}$ we have the separable closure $\bar{Q}$ of $Q$.  This gives a map $\Gal(\bar{Q}_{\v}/Q_{\v})  \ra \Gal(\bar{Q}/Q)$, and in particular  distinguished decomposition and  inertia groups  in $\Gal(\bar{Q}/Q)$ at ${\v}$ (as opposed to just a conjugacy classes of subgroups).  

As in Section~\ref{S:ff}, when $K\sub \bar{Q}$ with $K/Q$ a separable, quadratic extension, we let $K^{\un,{\v}}\sub\bar{Q}$ be the maximal extension of $K$ that is unramified everywhere and split completely at ${\v}$.
  We let $G_K^{\un,{\v}}:=\Gal(K^{\un,{\v}}/K).$  We note that in $\Gal(K^{\un,{\v}}/Q)$
the the inertia group at $\v$  has order dividing $2$ by Lemma~\ref{L:whinertia}.
 Thus if $K$ is ramified at ${\v}$, we have a distinguished non-trivial inertia element $i_{K,{\v}}\in \Gal(K^{\un,{\v}}/Q)$. 
As noted earlier, an automorphism that has order dividing $2$ is called an \emph{involution}.
 Conjugation by $i_{K,{\v}}$ gives an involution of $G_K^{\un,{\v}}$, and we let this conjugation be our chosen automorphism $\sigma$ of $G_K^{\un,{\v}}$.  (Note this is a more specific choice than we made in Section~\ref{S:ff} under different hypotheses.)
 %

Recall, for any finite group $H$ with an involution $\sigma$, we write $\Sur_\sigma(G_K^{\un,{\v}},H)$ for the continuous surjections taking conjugation by $i_{K,{\v}}$ to $\sigma$.
We  let $\tH=H \rtimes_\sigma \Ct$, and we denote the generator of $\Ct$ by $\sigma$ (a convenient overloading of notation).  
 Let $c$ be the set of elements of $\tH\setminus H$ of order $2$.

We define (as in \cite[Section 10.2]{EVW2}) a \emph{marked $(\tH,c)$ extension} of $Q$ to be $(L,\pi,m)$
such that 
$L/Q$ is a Galois extension of fields, 
$\pi$ is an isomorphism $\pi: \Gal(L/Q)\isom \tH$ such that all inertia groups in $\Gal(L/Q)$ (except for possibly the one at ${\v}$) have image in $\{1\}\cup c$, and $m$, the \emph{marking}, is a homomorphism $L_{\v}:=L\tensor_Q Q_{\v} \ra \bar{Q}_{\v}$.  Note that restriction to $L$ gives a bijection between homomorphisms 
$L_{\v} \ra \bar{Q}_{\v}$ and homomorphisms $L\ra \bar{Q}$.  Also, note that the condition that an inertia group in $\Gal(L/Q)$ has image in $\{1\}\cup c$ is equivalent to requiring that it has trivial intersection with $\pi^{-1}(H)$ because any element in $\tH\setminus (\{1\}\cup c)$ is either in $H$ or has square non-trivial in $H$. 
Two marked $(\tH,c)$ extensions $(L_1,\pi_1,m_1)$ and $(L_2,\pi_2,m_2)$ are isomorphic when there is an isomorphism $L_1\ra L_2$ taking $\pi_1$ to $\pi_2$ and $m_1$ to $m_2$.  
The marking $m$ in a marked $(\tH,c)$ extension $(L,\pi,m)$ gives a map $\Gal(\bar{Q}_{\v}/Q_{\v}) \ra \Gal(L/Q)$.
 Composing with $\pi$ we get an \emph{infinity type} $\Gal(\bar{Q}_{\v}/Q_{\v}) \ra \tH$. Such a homomorphism is called ramified if the image of inertia is nontrivial.
 
Note that in each isomorphism class of  marked $(\tH,c)$ extensions of $Q$, there is a distinguished element such that $L\sub \bar{Q}$ and $m|_{L}$ is the inclusion map.  

\begin{theorem}\label{T:trans}
Let $Q$ be a global field with a choice of place $\v$.
Let $H$ be a finite group with involution $\sigma$, let $\tH:=H \rtimes_\sigma \Ct$, and let
$c$ be the set of order $2$ elements of  $\tH \setminus H$.
%
Let $\phi: \Gal(\bar{Q}_\v/Q_\v) \ra \tH$ be a ramified homomorphism 
with image $\langle (1,\sigma) \rangle$.
There is a bijection between
$$
\{(K,f) | K\sub \bar{Q}, [K:Q]=2, K_{\v}/Q_{\v} \textrm{ the quadratic extension given by $\ker(\phi)$, } f\in \Sur_\sigma(G_K^{\un,\v},H) \}
$$ 
and
$$
\{\textrm{isomorphism classes of marked $(\tH,c)$ extensions $(L,\pi,m)$ of $Q$ with infinity type $\phi$}  \}.
$$

In this bijection, we have $\Disc(L)=\Disc(K)^{|H|}$.
\end{theorem}  



\begin{proof}
Given a $(K,f)$, we have that $\ker(f)$ gives a subfield of $L\sub K^{\un,v}\sub \bar{Q}$ and we have $f: \Gal(L/K) \isom H$.  We see that $\Gal(L/K)$ is an index $2$ subgroup of $\Gal(L/Q)$, and $i_{K,{\v}}$ is an order $2$ element of $\Gal(L/Q)\setminus \Gal(L/K).$  From the condition on the surjection $f$, we have that that $f$ takes the conjugation action of $i_{K,{\v}}$ on $\Gal(L/K)$ to the involution $\sigma$ on $H$.  Thus we can lift $f$ to
$\pi: \Gal(L/Q) \isom \tH$ with $i_{K,{\v}}\mapsto (1,\sigma)$.  We let the marking $m$ be the map $L_{\v}\ra \bar{Q}_{\v}$ induced by the identity on $L\sub \bar{Q}\sub\bar{Q}_{\v}.$  Since $L\sub K^{\un,\v}$, all inertia subgroups of $\Gal(L/Q)$ have image under $\pi$ in $\{1\}\cup c$.  
The infinity type $\Gal(\bar{Q}_{\v}/Q_{\v}) \ra \tH$ factors through the map $\pi$.
Since the index $2$ subgroup $\Gal(\bar{Q}_{\v}/K_{\v})$ has trivial image (it factors through $\Gal(L/K)$, and $L/K$ is split completely at $\v$), the infinity type of $m$ factors through the order $2$ group $\Gal(K_{\v}/Q_{\v})$.
Since, by construction of $\pi$, the inertia group $\Gal(\bar{Q}_{\v}/Q_{\v})$ has image $\langle (1,\sigma) \rangle,$ it follows that the infinity type is $\Gal(\bar{Q}_{\v}/Q_{\v})\ra \Gal(K_{\v}/Q_{\v}) \isom \langle (1,\sigma) \rangle$, which is $\phi$.

Given an isomorphism classes of marked $(\tH,c)$ extensions $(L,\pi,m)$ of $Q$ with infinity type $\phi$, we take the representative for which  $L\sub \bar{Q}$ and $m|_{L}$ is the identity map.  Then we let $K\sub \bar{Q}$ be the fixed field of $\pi^{-1}(H)$.  From the infinity type, we see that $L/K$ is split completely at $\v$, and that $K/Q$ is ramified at $\v$ such that $K_\v$ corresponds to $\ker(\phi)$.
By the fact that  $(L,\pi,m)$ is a $(\tH,c)$ extension of infinity type $\phi$, it follows that $L\sub K^{\un,\v}$, so we get a surjection $f: G_K^{\un,\v} \ra \Gal(L/K) \stackrel{\pi}{\ra} H$. 
From the infinity type, we see that $\pi$ takes  $i_{K,{\v}}\mapsto (1,\sigma)$, so we get that 
$f\in \Sur_\sigma(G_K^{\un,\v},H)$.

If we start with $(K,f)$, then by construction the fixed field of the $\pi^{-1}(H)$  from our constructed $(L,\pi)$ is $K$, and the restriction of $\pi$ to $\Gal(L/K)$ is $f$.  So if we apply both these constructions we return to  the same $(K,f)$. On the other hand, if we start with $(L,\pi,m)$
(such that $m$ is the identity), $L$ is the fixed field of the constructed morphism $f$, and $\pi$ is determined by the constructed $f$ and the image of $i_{K,{\v}}$, and so if we apply both these constructions we return to $(L,\pi,m)$.
\end{proof}

\section{Applying methods of Ellenberg-Venkatesh-Westerland to the extension counting problem}\label{S:EVW}

Theorem~\ref{T:FFqlimit} will follow  from Corollary~\ref{C:MainFF} in this section.
We will prove this result using a method and many results due to Ellenberg, Venkatesh, and Westerland in papers \cite{EVW1,EVW2}.  The method counts extensions of function fields by considering this as a problem of counting $\F_q$ points on a moduli space of curves with maps to $\P^1$, applying Grothendieck-Lefschetz to counts these points, and using results from topology to bound the dimensions of the cohomology groups.

\subsection{Group theory computation}
In this section, we will prove a lemma in group theory that will be central to proving Theorem~\ref{T:FFqlimit}.  This lemma will count $\F_q$-rational components in a moduli space on which we will eventually count points.  

First we will define the \emph{universal marked central extension} $\tilde{G}$ of a finite group $G$ for a union $c$ of conjugacy classes of $G$, following
\cite[Section 7]{EVW2}.
Let $\SC$ be a Schur cover of $G$ so we have an exact sequence
$$
1\ra H_2(G,\Z)\ra \SC \ra G \ra 1
$$
by the Schur covering map.
For $x,y\in G$ that commute, let $\hat{x}$ and $\hat{y}$ be arbitrary lifts to $\SC$, and let $\langle x,y\rangle$
be the commutator $[\hat{x},\hat{y}]\in \SC$, which actually lies in 
$ H_2(G,\Z)$ since $x$ and $y$ commute.  
  It we take the quotient of the above exact sequence by all $\langle x,y\rangle$ for $x\in c$ and $y$ commuting with $x$, we obtain an exact sequence 
$$
1\ra H_2(G,c)\ra \tilde{G}_c \ra G \ra 1,
$$
which is still a central extension. 
Let $G^{ab}$ denote the abelianization of $H$.
 The universal marked central extension is $\tilde{G}=\tilde{G}_c\times_{G^{ab}} \Z^{c/G}$, where $c/G$ denotes the set
of conjugacy classes in $c$ and the map 
 $\Z^{c/G} \ra G^{ab}$ sends
each standard generator to an element of the associated conjugacy class.  We have a map $\tilde{G} \ra G$, given through projecting to the first factor.  (See \cite[Section 7]{EVW2} for why this is called a universal marked central extension.)

\begin{lemma}\label{L:B}
Let $H$ be an odd finite group with a GI-automorphism $\sigma$, and $\tH=H\rtimes_\sigma \Ct$.
Let $c$ be the (single) conjugacy class of order $2$ elements.  Let $q$ be a power of a prime and $n$ be an odd integer. 
If $(q,2|H|)=1$ and $(q-1,|H|)=1$, then for each $y\in c$, there is exactly $1$ element 
$x\in \tilde{G}_c$ such that $(x,n)\in \tilde{G}$, and $x$ has image $y$ in $G$, and $x^q=x$.
\end{lemma}

\begin{proof}

We have that $|\tilde{\tH}_c|=2|H|||H_2(\tH,c)|$ and that $H_2(\tH,\Z)$ is a quotient of $H_2(H,\Z)$ by \cite[Example 9.3.2]{EVW2}.  Thus since $|H|$ is relatively prime to $2(q-1)$, we have that $|H_2(\tH,\Z)|$ is as well and thus $|H_2(\tH,c)|$ is as well.
 Since $|\tilde{\tH}_c|/2$ is relatively prime to $q-1$, we have that for $x\in \tilde{\tH}_c$, $x^q=x$ if and only if 
$x^2=1$. 

Let $w\in \tilde{\tH}_c$ be in the inverse image of $y$.  Then we ask for which $k\in H_2(\tH,c)$ is $wk$ of order $2$.  Since 
$H_2(G,c)\ra \tilde{\tH}_c$ is central, we have $(wk)^2=w^2k^2$, and note $w^2 \in H_2(\tH,c)$ since $y^2=1$.
Since $H_2(\tH,c)$ is an odd abelian group, there is exactly one $k\in H_2(\tH,c)$ such that $w^2k^2=1$.  Let $x=wk$ for this $k$, which is the only possible $x$ satisfying the conditions of  the lemma.   Also, note that $(x,n)\in \tilde{\tH}$ since $x$ and $n$ have image of the class of $y$ in $\tH^{ab}$, proving the lemma. 

%
%
 
 %
%
    
\end{proof}

\subsection{Properties of the Hurwitz scheme constructed by Ellenberg, Venkatesh, and Westerland}

In this theorem, we recall the Hurwitz scheme constructed by Ellenberg, Venkatesh, and Westerland to study extensions of $\F_q(t)$ and its properties.

\begin{theorem}[Ellenberg, Venkatesh, and Westerland]\label{T:Chur}
Let $H$ be an odd finite group with GI-automorphism $\sigma$, and let $\GG:=H \rtimes_\sigma \Ct$. 
 Let $c$ be the elements of $\GG$ of order $2$.  
 Let $\F_q$ be a finite field with $q$ relatively prime to $|\GG|$.
 When $\GG$ is center-free, there is a Hurwitz scheme $\CCHur_{G,n}$ over $\Z[|\GG|^{-1}]$ constructed in \cite[Section 8.6.2]{EVW2}\footnote{ The paper \cite{EVW2} has been temporarily withdrawn by the authors
 because of a gap which affects Sections 6, 12 and some theorems of the introduction of \cite{EVW2}.  That gap does not affect any of the results from \cite{EVW2} that we use in this paper.
 }  with the following properties:
 \begin{enumerate}

\item We have $\CCHur_{G,n}$ is a finite \'etale cover of the relatively smooth $n$-dimensional configuration space $\Conf^n$ of $n$ distinct unlabeled points in $\A^1$ over $\Spec \Z[|\GG|^{-1}]$.
\label{i:smooth} 
 
 \item  \label{i:bij} 
 The scheme $\CCHur_{G,n}$ has an open and closed subscheme $\CCHur^{c,c}_{G,n}$ such that there is a bijection between
 \begin{enumerate}
 \item isomorphism classes of  marked $(\GG,c)$-extensions $L$ of $\F_q(t)$ of 
 $\Nm \Disc (L)=q^{(n+1)|H|}$
 and an infinity type $\phi$ such that $\phi(F_\Delta)=1$ and $\im \phi$ is of order $2$ and in $c \cup \{1\}$ (where $F_\Delta$ is a lift of Frobenius to $\Gal(\bar{Q}_{\v}/Q_{\v})$ that acts trivially on $\F_q((t^{-1/\infty}))$).
%
 \item points of $\CCHur^{c,c}_{G,n}(\F_q)$ \cite[Section 10.4]{EVW2}.
 \end{enumerate}

\item We have $\CCHur_{G,n}(\C)$ is homotopy equivalent to a topological space $\CHur_{G,n}$ \cite[Section 8.6.2]{EVW2}, 
such that 
for any field $k$ of characteristic relatively prime to $|\GG|$,
there is a constant $C$ such that for all $i\geq 1$ and for all $n$ we have $\dim H^i(\CHur_{G,n},k)\leq C^i$ \cite[Proposition 2.5 and Theorem 6.1]{EVW1}. \label{i:topbounds}

%
%

\item Given $G$, for $n$ sufficiently large and all $q$ with $(q,G)=1$, the $\Frob$ fixed components of $\CCHur^{c,c}_{G,n}  \tensor_{\Z[|\GG|^{-1}]} {\bar{\F}_q}$ are in bijection with elements
$(x,n) \in \tilde{G} $ such that $x^q=x$ and $x$ has image of order $2$ in $G$  \cite[Theorem 8.7.3]{EVW2}. (The requirement that $x$ has image of order $2$ in $G$ ensures the monodromy at $\v$ is in $c$.)
\label{i:comp}

 \end{enumerate}
\end{theorem} 

\begin{remark}
The scheme $\CCHur^{c,c}_{G,n}\sub \CCHur_{G,n}$ comes from restricting to the parametrization of covers of $\P^1$  all of whose local inertia groups have image 
 in $c \cup \{1\}$.  We use two $c$ superscripts because \cite{EVW2} uses a single $c$ superscript to denote when this restriction is made only over points in $\A^1\sub \P^1$.  The argument that $\CCHur^{c,c}_{G,n}\sub \CCHur_{G,n}$ is an open and closed subscheme is as in \cite[Section 7.3]{EVW1}. 
Our description of the components requires a bit of translation from that in \cite[Theorem 8.7.3]{EVW2}.  They biject the components with $\hat{\Z}^\times$ equivariant functions from topological generators of 
$\varprojlim \mu_n$ (taken over $n$ relatively prime to $q$) to the preimage of $c$ in $\tilde{G}$ that are fixed by the discrete action of $\Frob$.  By choosing any topological generator of $\varprojlim \mu_n$, its image under a function to $\tilde{G}$ gives us a corresponding element of $\tilde{G}$.   Using the definition of the discrete action and \cite[Equation (9.4.1) and 9.3.2]{EVW2}, we can see that under this correspondence $(x,n)\mapsto (x^q,n)$ describes the inverse of $\Frob$.
\end{remark}
 
 \subsection{Counting $\F_q$ points}
In this section, we will count the $\F_q$ points of $\CCHur^{c,c}$ in Theorem~\ref{T:count}, and then use our Theorem~\ref{T:trans} to translate that into a result about surjections from Galois groups $G_K$ in Corollary~\ref{C:MainFF}, which will finally prove Theorem~\ref{T:FFqlimit}.

\begin{theorem}\label{T:count}
Given $G$ and $c$ as in Theorem~\ref{T:Chur},
 we have a constant $C$ and a constant $n_G$ such that for $q>C^2$, with $(q,|G|)=1$ and $(q-1,|G|/2)=1$, and odd $n\geq n_G$,
$$
|\#\CCHur^{c,c}_{G,n}({\F}_q) - q^{n} \cdot \#c|\leq  \frac{q^{n}}{\sqrt{q}/C-1}.
$$
\end{theorem}
\begin{proof}
Our theorem will follow by applying the Grothendieck-Lefschetz trace formula to $X:= \CCHur^{c,c}_{G,n} \tensor_{\Z[|\GG|^{-1}]} {{\F}_q}$.
 By Theorem~\ref{T:Chur} \eqref{i:smooth}, we have that $X$ is smooth of dimension $n$.
 We have that $\dim H^i_{\textrm{c,\'et}}(X_{\bar{\F}_q},\Q_\ell)=
 \dim H^{2n-i}_{\textrm{\'et}}(X_{\bar{\F}_q},\Q_\ell)$ by Poincar\'{e} Duality.
 
Next, we will relate $\dim H^{j}_{\textrm{\'et}}(X_{\bar{\F}_q},\Q_\ell)$ to 
 $\dim H^{j}(\CCHur^{c,c}_{G,n}(\C),\Q_\ell)$ for some $\ell>n$.  To compare \'{e}tale cohomology between characteristic $0$ and positive characteristic, we will use \cite[Proposition 7.7]{EVW1}.  The result \cite[Proposition 7.7]{EVW1} gives an isomorphism between \'{e}tale cohomology between characteristic $0$ and positive characteristic in the case of a finite cover of a complement of a reduced normal crossing divisor in a smooth proper scheme.  Though \cite[Proposition 7.7]{EVW1} is only stated	 for \'{e}tale cohomology with coefficients in $\Z/\ell\Z$, the argument goes through identically for coefficients in $\Z/\ell^k\Z$, and then we can take the indirect limit and tensor with $\Q_{\ell}$ to obtain the result of \cite[Proposition 7.7]{EVW1} with $\Z/\ell\Z$ coefficients replaced by $\Q_\ell$ coefficients.  So we apply this strengthened version to conclude that 
$ \dim H^{j}_{\textrm{\'et}}(X_{\bar{\F}_q},\Q_\ell)=\dim H^{j}_{\textrm{\'et}}((\CCHur^{c,c}_{G,n})_{\C}),\Q_\ell)$.
(As in \cite[Proof of Proposition 7.8]{EVW1}, we apply comparison to $\CCHur^{c,c}_{G,n} \times_{\Conf^n } \operatorname{PConf}_n,$ where $\operatorname{PConf}_n$ is the moduli space of $n$ labelled points on $\A^1$, and is the complement of a relative normal crossings divisor in a smooth proper scheme \cite[Lemma 7.6]{EVW1}.
Then we take $S_n$ invariants to compare the \'etale cohomology of $\CCHur^{c,c}_{G,n}$ across characteristics.)
By the comparison of \'{e}tale and analytic cohomology \cite[Expos\'e XI, Theorem 4.4]{SGA4:3} $\dim H^{j}(\CCHur^{c,c}_{G,n}(\C),\Q_\ell)=\dim H^{j}_{\textrm{\'et}}((\CCHur^{c,c}_{G,n})_{\C}),\Q_\ell)$.  

By Theorem~\ref{T:Chur}~\eqref{i:topbounds}, there is a constant $C$ such that for all $j\geq 1$ and for all $n$, we have $\dim H^{j}(\CCHur^{c,c}_{G,n}(\C),\Q_\ell)\leq C^j$.  Thus $ \dim H^{j}_{\textrm{\'et}}(X_{\bar{\F}_q},\Q_\ell)\leq C^j$ for all $j\geq 1$.
  Thus using Poincar\'e duality, 
$ \dim H^{i}_{\textrm{\'et},c}(X_{\bar{\F}_q},\Q_\ell)\leq C^{2n-i}$ for all $i<2n$.
By Theorem~\ref{T:Chur} \eqref{i:comp} and Lemma~\ref{L:B}, we have that $X$ has $\#c$ components fixed by $\Frob$ for odd $n\geq n_G$ for some fixed $n_G$.

Then by the Grothendieck-Lefschetz trace formula we have
$$
\#X({\F}_q) =\sum_{j\geq 0} (-1)^j\Tr(\Frob|_{ H^j_{\textrm{c,\'et}}(X_{\bar{\F}_q},\Q_\ell)})
$$
and 
also we know $\Tr(\Frob|_{ H^{2n}_{\textrm{c,\'et}}(X_{\bar{\F}_q},\Q_\ell)})$ is $q^{n}$ times the number of components of $X$ fixed by $\Frob$.
Since $X$ is smooth, we have that the absolute value of any eigenvalue of $\Frob$ on $H^j_{\textrm{c,\'et}}((X_{\bar{\F}_q},\Q_\ell)$ is at most
$q^{j/2}$  .
Thus, for odd $n\geq n_G$,
\begin{align*}
|\#X({\F}_q)  - q^{n} \times\#c|&=\left| \sum_{0\leq j< 2\dim X} (-1)^j\Tr(\Frob|_{ H^j_{\textrm{c,\'et}}((X_{\bar{\F}_q},\Q_\ell)}) \right| \\
&\leq \sum_{0\leq j< 2\dim X} q^{j/2} C^{2n -j}\\
&\leq q^{n} \sum_{1\leq i } (\sqrt{q}/C)^{-i} .
\end{align*}
 The theorem follows.
\end{proof}

We have $Q=\F_q(t)$ and $Q_\infty=\F_q((t^{-1}))$, for $q$ odd. 
Unlike in the number field case, in which there is only one possible ramified quadratic extension of $\Q_\infty=\R$, here there are two ramified quadratic extensions of $Q_\infty=\F_q((t^{-1}))$. 
 If $K/\F_q(t)$ is a quadratic extension, we say it is imaginary quadratic of type I if $K_\infty\isom  \F_q((t^{-1/2}))$ and of type II if $K_\infty\isom  \F_q(((\alpha t)^{-1/2}))$ for an $\alpha\in \F_q\setminus \F_q^2$. 
Let $IQ_n$ be the set of $K\sub \bar{Q}$ such that $K$ is imaginary quadratic of type I and  $\Nm \Disc (K)=q^{n+1} $.  Let $IQ'_n$ be the set of $K\sub \bar{Q}$ such that $K$ is imaginary quadratic of type II and  $\Nm \Disc (K)=q^{n+1} $.

\begin{corollary}\label{C:MainFF}
Let $H$ be an odd finite group with GI-automorphism $\sigma$ such that $H\rtimes_\sigma \Ct$ is center-free.
As $q$ ranges through powers of primes such that $(q,2|H|)=1$ and $(q-1,|H|)=1$, we have
$$
\lim_{q\ra\infty} \limsup_{\substack{n\ra\infty\\ n \textrm{ odd}}} \frac{\sum_{K\in IQ_n} |\Sur_\sigma(G_K^{\un,\v},H)|}
{ \#IQ_n
}
=1.
$$
The same result holds if we replace $\limsup$ by $\liminf$ and/or  replace $IQ_n$ by $IQ'_n$.
\end{corollary}

Theorem~\ref{T:FFqlimit} then follows from Corollary~\ref{C:MainFF} after noting that $H\rtimes_\sigma \Ct$ is center-free if and only if the center of $H$ contains no elements fixed by $\sigma$ except the identity.  

\begin{proof}
By Theorem~\ref{T:Chur}~\eqref{i:bij} the points $\CCHur^{c,c}_{G,n}(\F_q)$ are in bijection with isomorphism classes of marked $(\GG,c)$ extensions $(L,\pi,m)$ of $Q$ with certain infinity types $\phi$.  These infinity types are all $\GG$-conjugate, and there are $\#c$ of them.
Let $\phi_0$ be the infinity type such that $\phi(F_\Delta)=1$ and $\im \phi=\langle (1,\sigma) \rangle$.
 Note that $\F_q((t^{-1/2}))$ is the imaginary quadratic extension given by $\ker(\phi_0)$.

  Let $\phi: \Gal(\bar{Q}_\v/Q_\v) \ra \GG$ be a ramified homomorphism with image $\langle (1,\sigma) \rangle$,  let $g\in \GG$, and let $\phi^g$ denote the conjugation. Then 
isomorphism classes of marked $(\GG,c)$ extensions $(L,\pi,m)$ of $Q$ with infinity type $\phi$ of a given discriminant
are in bijection with   isomorphism classes of marked $(\GG,c)$ extensions $(L,\pi,m)$ of $Q$ with infinity type $\phi^g$ and that discriminant by sending $(L,\pi,m)$ to $(L,\pi^g,m)$.
So, we have that 
\begin{align*}
&\#\CCHur^{c,c}_{G,n}(\F_q)\\
&=\#c \cdot \#\{ &\textrm{isom. classes of  marked $(\GG,c)$-extensions $L/\F_q(t)$ of infinity type $\phi_0$}\\ & & \textrm{ and $\Nm \Disc (L)=q^{(n+1)|H|}$}\}.
\end{align*}
Further, by Theorem~\ref{T:trans}, we then conclude that 
\begin{align*}
&\#\CCHur^{c,c}_{G,n}(\F_q)\\
&=\#c \cdot \{(K,f) |  K\sub \bar{Q}, K \textrm{ imag. quad. type I, }
f\in \Sur_\sigma(G_K^{\un,\v},H),
 \Nm \Disc (K)=q^{n+1} \}.
\end{align*}

So by  Theorem~\ref{T:count}, we have a constant $C$, only depending on $H$, such that for $q\geq 4C^2$ and odd $n\geq n_G$ 
$$
\left| {\sum_{K\in IQ_n} |\Sur_\sigma(G_K^{\un,\v},H)|} -q^n \right|
\leq {2Cq^{n-1/2}}.
$$
  Thus, for $q\geq 4C^2$ and all odd $n\geq n_G$ 
$$
\frac{\sum_{K\in IQ_n} |\Sur_\sigma(G_K^{\un,\v},H)|}
{\#IQ_n
}
=\frac{q^n +O(q^{n-1/2})}{q^n-q^{n-1}}=1+O(q^{-1/2}).
$$
It follows that the limit as $q\ra \infty$, of the of $\limsup$ or $\liminf$, in odd $n$, of the lefthand side are both 1.
 For the case of $IQ'_n$, we have a bijection $K\mapsto K \tensor_{\F_q(t)}\F_q(t)$  (where the map $\F_q(t) \ra \F_q(t)$ is given by $t \mapsto \alpha t$, for some $\alpha\in \F_q\setminus \F_q^2$) between $IQ_n$ and $IQ_n'$ that preserves $G_{K}^{\un,\v}.$
\end{proof}

\subsection{Further results assuming a conjecture on the homology of Hurwitz spaces}

The program developed by Ellenberg, Venkatesh, and Westerland in \cite{EVW2} aims to prove stronger results on the topology of Hurwitz spaces, from which corresponding stronger results on the point counts would follow.  
For example, $\operatorname{HS}_\alpha$ \cite[Section 11.1]{EVW2} is a conjecture on the homology of Hurwitz spaces for a given group $G$ and conjugacy invariant subset $c$.  

\begin{theorem}\label{T:withconj}
Let $H$ be an odd finite group with GI-automorphism $\sigma$ such that $H\rtimes_\sigma \Ct$ is center-free.  If $\operatorname{HS}_\alpha$ holds for $G=H\rtimes_\sigma \Ct$ and $c$ the order $2$ elements of $G$, 
then there is a $q_0$ such that for $q\geq q_0$, with $(q,2|H|)=1$ and $(q-1,|H|)=1$, we have
$$
\limsup_{\substack{n\ra\infty\\ n \textrm{ odd}}} \frac{\sum_{K\in IQ_n} |\Sur_\sigma(G_K^{\un,\v},H)|}
{ \#IQ_n
}
=1.
$$
The same result holds if we  replace $IQ_n$ by $IQ'_n$.
\end{theorem}

\begin{proof}
We apply Theorem~\ref{T:trans} and \cite[Theorem 11.1.1]{EVW2}.  Lemma~\ref{L:B} shows that the quantity $B(L_\infty, \mathfrak{m})$ appearing in \cite[Theorem 11.1.1]{EVW2} is $1$.  Finally, we use that an \'etale $G$-extension $L_\infty$  has $|G|/|\Aut_G(L_\infty)|$ corresponding infinity types and a $G$-extension has  has $|G|$ markings.  
\end{proof}

\section{Non-equivariant moments}\label{S:noneq}
While in this paper, we have asked about the equivariant moments, or averages of $|\Sur_\sigma(G_K^{\un,\v},H)|$, one could naturally ask about non-equivariant moments, or averages of $|\Sur(G_K^{\un,\v},H)|$.  It turns out these non-equivariant moments reduce in a simple way to equivariant moments.

Let $G$ be a group with a GI-automorphism $\sigma$.  Then we have an injection
\begin{align*}
\Sur(G,H) &\ra \Hom_\sigma (G,H\times H)\\
f &\mapsto f \times f\sigma,
\end{align*}
where the automorphism $\sigma$ of $H\times H$ is switching the factors.
In fact, this is a bijection onto the subset of $\Hom_\sigma (G,H\times H)$ that surject onto the first factor.
Let $\mathcal{F}$ be the set of $\sigma$-invariant subgroups of $H\times H$ that surject onto the first factor.
Then
\begin{equation}\label{E:noneq}
|\Sur(G,H)| = \sum_{F\in \mathcal{F}} |\Sur_\sigma(G,F)|.
\end{equation}
Note since $\sigma$ is GI on $G$, if it is not $GI$ on $F$, then $|\Sur_\sigma(G,F)|=0$.  Thus Equation~\eqref{E:noneq} would still hold if we restrict the sum on the right to $F$ such that switching factors in $H\times H$ is $GI$ on $F$ (i.e. $F$ generated by elements of the form $(h,h^{-1})$ for $h\in H$).

\subsection*{Acknowledgements} 
The authors thank Jordan Ellenberg, Daniel Ross, and Akshay Venkatesh for helpful conversations.
We thank Bjorn Poonen for pointing out a mistake in an earlier version of Lemma~\ref{L:la} as well as providing a quicker proof of the lemma.
The first author was supported by Simons Foundation Award MSN-179747. The second author was supported by an American Institute of Mathematics Five-Year Fellowship, a Packard Fellowship for Science and Engineering, a Sloan Research Fellowship, and National Science Foundation grants DMS-1147782 and DMS-1301690.

\newcommand{\etalchar}[1]{$^{#1}$}
\def\cprime{$'$}

\end{document}